\documentclass[12pt]{article}

\usepackage{amsmath, amsfonts,amssymb,  amscd, amsthm, verbatim, mathtools} %, amsxtra}
\usepackage{fullpage}

\def\L{\mathcal{L}}
\def\A{\mathbb{A}}

\def\G{\mathbb{G}}

\def\P{\mathbb{P}}

\def\onto{\twoheadrightarrow}
\def\into{\hookrightarrow}

\def\F{\mathbb{F}}

\theoremstyle{plain}

\newtheorem{thm}{Theorem}

\newtheorem{Lem}[thm]{Lemma}

\newtheorem{Prop}[thm]{Proposition}

\newtheorem{Def}[thm]{Definition}

\theoremstyle{remark}
\newtheorem{Rem}[thm]{Remark}

\begin{document}

\title{Variants of the Kakeya problem over an algebraically closed field\footnote{
The final publication in Archiv der Mathematik is available at Springer via 
http://dx.doi.org/[10.1007/s00013-014-0685-6]
}}
\author{Kaloyan Slavov}
\maketitle
\begin{abstract}
First, we study constructible subsets of $\A^n_k$ which contain a line in any direction. We classify the smallest such subsets in $\A^3$ of the type $R\cup\{g\neq 0\},$ where $g\in k[x_1,...,x_n]$ is irreducible of degree $d$, and $R\subset V(g)$ is closed. Next, we
study subvarieties $X\subset\A^N$ for which the set of directions of lines contined in $X$ has the maximal possible dimension. 
These are variants of the Kakeya problem in an algebraic geometry context.  
\end{abstract}

\section{Introduction}

In \cite{W}, T. Wolff proposed a finite field model for the classical harmonic analysis Kakeya problem. 
Namely, he defines a Kakeya subset $E$ of $\F_q^n$ to be a subset which contains a line in any direction, in analogy with the notion of a Kakeya subset of 
$\mathbb{R}^n$, which is a compact subset containing a unit line segment in any direction.
The finite field Kakeya problem was solved by Z. Dvir in \cite{Dvir} and has proved to be a useful model for the hard classical Euclidean problem.  Recently in \cite{DH}, Dummit and Hablicsek answered a question of Ellenberg, Oberlin, and Tao about Kakeya subsets of
$\F_q[[t]]^n$; this is a version of the Kakeya problem over non--archimedean local rings. 

We give two different {\it algebraic--geometry} versions of the Kakeya problem that are interesting over an algebraically closed field $k$, of any characteristic. Our main motivation is that the smallest known example of a Kakeya subset of $\F_q^n$ arises from a Kakeya variety as in Definition \ref{def_linear_Kakeya_cover}. Thus, this extra structure of an algebraic variety coming with the smallest known example of a Kakeya subset of $\F_q^n$ should not be neglected, and studying it presents sufficient motivation and independent interest. At the very least, this leads to interesting algebraic geometry questions and structural results. More importantly, algebraic geometry models tie with the general philosophy and metatheorem that extreme combinatorial configurations posses algebraic structure. The classical harmonic analysis Kakeya problem is notoriously difficult; on the other hand, the algebraic geometry version that we present here is approachable due to the rich structure coming with the hypothesis of constructibility.  

In Section  \ref{weak_Kakeya_section},
we study a constructible subset $E$ of $\A^n$
which contains a line in any direction. In analogy with classical frameworks, we call such an $E$ a Kakeya subset of $\A^n$. 
The starting point of our investigation in Section \ref{geom_Kakey_vars_section} is

\begin{Prop}
{\bf a)}\ If $E\subset\A^n_k$ is a constructible Kakeya subset, then $\dim E=n.$

{\bf b)}\ Let $E\subset\A^n_k$ be an open subset. Then $E$ is a Kakeya subset if and only if $\dim E^c\leq n-2$. 
\label{dim_Kakeya}
\end{Prop}

 Consider a constructible Kakey subset $E\subset\A^n_k$, where $E$ is a finite disjoint union 
$E=\cup T_i$ of locally closed
subsets $T_i$ of $\A^n_k.$ Each $T_i$ is an open subset of a closed subset $\hat{T}_i$ of $\A^n$,
and for some $i$ (uniquely determined), $U:=T_i$ is open
in $\A^n$ by Proposition \ref{dim_Kakeya}a. 
 So, 
$E=R_1\cup\dots\cup R_s\cup U,$
 where each $R_i$ is locally closed in $\A^n$ with $\dim \overline{R_i}\leq n-1$, and $U$ is a nonempty open in $\A^n$. 
 
If $\dim U^c\leq n-2,$ then $U$ itself is a Kakeya set by Proposition \ref{dim_Kakeya}b. 

Consider now the ``small" case when $\dim U^c=n-1.$ 
So, $U^c$ is a union of finitely many irreducible hypersurfaces, together with some lower--dimensional irreducible components. We say that $E$ is of type $t$ if $U^c$ has exactly $t$ irreducible components, all of them of dimension $n-1$. The larger the $t$, the smaller the $U$. We focus on the case $t=1$ and consider constructible Kakeya subsets of $\A^n$ of type $1$; our goal is to describe the smallest such subsets. 

We now study the decomposition
\begin{equation}
E=R_1\cup\dots\cup R_s\cup\{g\neq 0\}
\label{form_Kakeya}
\end{equation}
(where $g\in k[x_1,...,x_n]$ is irreducible) of a Kakeya subset of type $1$, 
in terms of the degree $d$ of $g$. A larger degree $d$ will correspond to a smaller Kakeya set.
Further, for a fixed $d$, a small Kakeya set will have $R_1\cup\dots\cup R_s$ small (the first measure will be its dimension). 

In Section \ref{examples_section}, we give an extreme example of a constructible Kakeya subset of $\A^3$ of type $1$:

\begin{Prop}
For any $d\geq 3$, there exists an irreducible hypersurface $V(g)\subset\A^3_k$ of degree $d$, and $2$ points $R_1,R_2$ on it, such that $\{R_1,R_2\}\cup \{g\neq 0\}$ is a Kakeya subset of $\A^3_k.$
\label{example_cone_3_points}
\end{Prop}

The main result of Section \ref{uniqueness_cone_section} is that the construction in the proof of Proposition \ref{example_cone_3_points} is essentially the only example of such an extreme small Kakeya subset of $\A^3$ of type $1$; to state it precisely, we introduce some notation.
For a subvariety $X\subset\A^n$, we denote by $\overline{X}$ the
projective closure of $X$ in $\P^n$. The hyperplane at infinity in $\P^n$ is denoted by $V(x_0)$, so $x_1,...,x_n$ will be affine coordinates in $\A^n$, while $[x_0:x_1:...:x_n]$ will be projective coordinates in $\P^n$. For a polynomial $g\in k[x_1,...,x_n]$,
we denote by $G$ its homogenization with respect to $x_0,$ so $\overline{V(g)}=V(G).$

\begin{Prop}
Let $E=\{R_1,\dots, R_s\}\cup\{g\neq 0\}\subset\A^3_k$ be a Kakeya subset, where $g$ is irreducible of degree $d\geq 3,$ and $R_1,...,R_s\in V(g).$ Let $C=V(G)\cap V(x_0)$.
Then $C$ is irreducible, and $V(G)$ is a cone over it. Moreover, if $C$ is non--flexy and non-funny\footnote{See Section \ref{uniqueness_cone_section} for the definitions.}, then $C$ has a unique singular point, whose multiplicity is $d-1$.  
\label{cone_uniqueness_prop}
\end{Prop}

In particular, given an irreducible
hypersurface $V(g)\subset\A^3,$ unless $g$ satisfies the very stingy requirements described in the statement above, it is not possible to add finitely many points $R_1,...,R_s$ to $V(g)$ so that $\{R_1,...,R_s\}\cup\{g\neq 0\}$ is a Kakeya subset. 

Next, in Section \ref{Kakeya_subs_section}, we
consider a subvariety $X\subset\A^N$ with $\dim X=n$. Let $V(x_0)=\P^N-\A^N$ be the hyperplane at infinity, and let 
$\Delta\subset V(x_0)$ be the set of all directions of lines contained in $X$. We say that $X$ is a Kakeya subvariety of $\A^N$ if the
inequality $\dim\Delta\leq n-1$ is an equality. As an example, we prove that any hypersurface $X\subset\A^N$ of degree $d<N$ is Kakeya in this sense. Next, we
propose the following

\begin{Def}
A subvariety $X\subset\A^N_k$ of dimension $n$, together with a morphism $\pi:X\to\A^n_{z_1,...,z_n}$ is called a Kakeya cover if there exists an open subset $U\subset V(z_0)=\P^n_{[z_0:...:z_n]}-\A^n_{z_1,...,z_n}$ with the following property: for any $v\in U,$ there
exists a line $l\subset X\subset \A^N_k$ whose image under $\pi$ is a line in $\A^n$ in direction $v$. 
\label{def_linear_Kakeya_cover}
\end{Def}

The smallest known example of a combinatorial Kakeya subset of $\F_q^n$ comes from
\[\{(a_1,...,a_{n-1},b)\in\F_q^n \ |\ a_i+b^2\ \text{is a square in $\F_q$ for all $i$}\}\subset\F_q^n   \]
(say $q$ is odd, for convenience); see Appendix A in \cite{SarafSudan}. 
This is the image on $\F_q$-points of the composition
\[
X=V\left(a_1+b^2-c_1^2,...,a_{n-1}+b^2-c_{n-1}^2\right)
\into \A^{2n-1}_{a_1,...,a_{n-1},b,c_1,...,c_{n-1}}\to \A^n_{a_1,...,a_{n-1},b}
\]
which is a Kakeya cover as in Definition \ref{def_linear_Kakeya_cover}. Indeed, take $U=\{b\neq 0\}$ and for any direction $v=[\alpha_1:...:\alpha_{n-1}:1]\in U$, consider the line
$a_i=\alpha_i t+\frac{\alpha_i^2}{4}, b=t, c_i=t+\frac{\alpha_i}{2}.$ It is contained in $X$, and its image under $X\to\A^n$ is a line in direction $v$.
This justifies the significance of Definition \ref{def_linear_Kakeya_cover}. 

It is easy to prove the following 

\begin{Prop}
Let $X\subset\A^N$ be an irreducible $n$-dimensional Kakeya subvariety of $\A^N$, and let $d$ be the degree of its closure $\overline{X}\subset\P^N$. After performing $\text{codim} X$ appropriate linear projections, we obtain a finite map $\pi: X\to\A^n_{z_1,...,z_n}$ of degree $d$, which is a Kakeya cover. 
\label{Kakeya_sub_gives_cover_Prop}
\end{Prop}

In \cite{Slavov_alg_geom_version}, we study covers which satisfy a more restrictive version of  Definition \ref{def_linear_Kakeya_cover}.

Sections \ref{weak_Kakeya_section} and \ref{Kakeya_subs_section} are independent of one another and present different viewpoints towards a Kakeya problem in the context of algebraic geometry. 
Throughout the article, $k$ is a fixed algebraically closed field.

\section{Constructible Kakeya subsets}
\label{weak_Kakeya_section}

\subsection{The geometry of constructible Kakeya subsets}
\label{geom_Kakey_vars_section}

\begin{proof}[Proof of Proposition \ref{dim_Kakeya}a.]
Suppose that $\dim E\leq n-1$. Replacing $E$ by its closure, we can assume without loss of generality that $E\subset\A^n_k$ is a closed subset. Let $V(x_0)=\P^n-\A^n$ be the hyperplane at infinity; the direction of a line in $\A^n$ is a point in $V(x_0)$.  
Let $v\in V(x_0)$ be arbitrary. We know that there exists some line $l$ in $\A^n_k$ which is contained in $E$ and whose projective closure $\overline{l}$ passes through $v$. Taking closures in $\P^n_k,$ the inclusion $l\subset E$ implies $\overline{l}\subset\overline{E}$ hence $v\in\overline{E}$. Since $v$ was arbitrary, we deduce that $V(x_0)\subset \overline{E}$. However, since $\dim \overline{E}\leq n-1,$ this is possible only if $\dim\overline{E}=n-1,$ and $V(x_0)$ is one of its irreducible components. This is impossible, since $E\subset\A^n_k,$ and hence cannot be a dense subset of $V(x_0)$. 
\end{proof}

\begin{Rem}
We can compare and parallel the above proof with Dvir's proof of the Kakeya problem in the finite field setting. Here $\overline{E}$ plays the role of the hypersurface in Dvir's proof. The requirement $\deg f=d<q$ was needed in Dvir's work to derive a contradiction from $V(x_0)\subset V(\overline{f})$. Here in the geometrical setting, this is automatic. So, our proof of Proposition \ref{dim_Kakeya}a is a geometric version of Dvir's argument.  
\end{Rem}

\begin{proof}[Proof of Proposition \ref{dim_Kakeya}b.]
Let $Z=\A^n_k-E$. Suppose first that $\dim Z\leq n-2$. To show that $E$ is a Kakeya subset, consider an arbitrary $v\in V(x_0)$. 
Let $\G_v$ be the subset of the Grassmanian $\G(1,n)$ consisting of all lines in $\P^n$ passing through $v$. Consider the incidence correspondence
\[I:=\{ (x,l)\in Z\times\G_v\ |\ x\in l\}\subset Z\times\G_v,\]
together with its two projections to $Z$ and $\G_v$.
Each fiber of $I\to Z$ is $0$-dimensional, so $\dim I=\dim Z\leq n-2.$ Since $\dim\G_v=n-1,$ there is a dense open subset of $\G_v$ where every point is outside of the image of $I\to\G_v$. Any such line will be entirely contained in $E$. 

Conversely, suppose that $\dim Z=n-1$. We claim that $E$ is not a Kakeya subset. Let $\overline{Z}$ be the closure of $Z$ in $\P^n$. Note that 
$V(x_0)\not\subset \overline{Z}$.
 Let $v\in V(x_0)-\overline{Z}$. If 
 $\overline{l}=l\cup\{v\}\in\G_v$, then $\overline{l}$ must intersect $\overline{Z}$ at some point $w$, by Bezout's theorem. Since $w\neq v$, we have $w\in\overline{Z}\cap\A^n=Z$, and so $l$ is not contained in $E$.  
\end{proof}

\subsection{Examples of small Kakeya subsets of type $1$}
\label{examples_section}

When $d=1$, so $g=L$ is linear, the smallest such Kakeya set would have to be $\{q\}\cup\{L\neq 0\}$ where $q$ is a point on the hyperplane $V(L)$. In fact, for any $q\in V(L),$ the subset $E=\{q\}\cup\{L\neq 0\}$ is a Kakeya subset of $\A^n$.

\begin{Prop}
For any $d\geq 1$, there exists a Kakeya set $E\subset\A^n_k$ of the form $E=R\cup \{g\neq 0\},$ where $g\in k[x_1,...,x_n]$ is irreducible of degree $d$, and
$R$ is a closed subset of $V(g)$ of dimension $n-2.$ 
\label{example_dim_n_2}
\end{Prop}  

\begin{proof}
Consider a smooth irreducible hypersurface $C\subset V(x_0)\simeq \P^{n-1}$ of degree $d$, and let $X=V(G)\subset\P^n$ be the projective cone over it, with vertex $v=[1:0:...:0].$  
Choose a point $v'\in\P^n-X-V(x_0)$ and let $X'$ be the cone over $C$ but with vertex $v'.$ Let $R=X\cap X'$; it is an $(n-2)$-dimensional closed subvariety of $X$. Then
$E=\left(\{v\}\cup R\cup \{G\neq 0\}\right)\cap \A^n\subset\A^n_k$ is a Kakeya subset of $\A^n_k$. 
\end{proof}

\begin{Lem}
Let $X\subset\P^3_k$ be the cone over some irreducible curve $C\subset V(x_0)$, with vertex $v=[1:0:0:0]$.
 Suppose that $C$ is not a line (equivalently, that $X$ is not a plane). If a line $l\subset\P^3$ is contained in $X$, then the vertex $v$ of $X$ belongs to $l$.
\label{triv_vertex}
\end{Lem}

\begin{proof}
Consider the projection map $\P^3-\{v\}\to V(x_0)$. If $l$ does not pass through $v$, then its image is well--defined and is a line $l'$ in $V(x_0)\simeq\P^{2}.$ By the definition of a cone
as the union of all lines connecting points of $C$ with $v$, we must have $l'\subset C$. By the irreducibility of $C$, this would yield $l'=C$, which is a contradiction to the hypothesis.
\end{proof}

When $d=2$, so $g=Q$ is irreducible quadratic, again a Kakeya type-$1$ set $E=R_1\cup\dots\cup R_s\cup \{Q\neq 0\}$ in $\A^3$ would have to contain at least one point of $V(Q)$. In fact, adding just one point suffices, if $Q$ is chosen appropriately: 
we now give an example of such a Kakeya set $E=\{q\}\cup\{Q\neq 0\}$. Let $C$ be a smooth quadric hypersurface in $V(x_0)\simeq\P^{2}$ defined by $Q\in k[x_1,...,x_3],$
 and let $V(Q)\subset\P^3$ be the cone over $C$, with vertex at $v=[1:0:0:0].$ Then $E=\{v\}\cup\{Q\neq 0\}$ is a Kakeya subset of $\A^3$. 
 
One naturally asks for a refinement of this example. For a given $d$, we want to find an irreducible hypersurface $V(g)\subset\A^3_k$ of degree $d$ and a subset $R\subset V(g)$ of dimension as small as possible, so that $R\cup\{g\neq 0\}$ is a Kakeya subset of $\A^3_k$. 
Indeed, an extreme type-$1$ Kakeya subset of $\A^3$ as described in Proposition \ref{example_cone_3_points} exists, for any $d\geq 3.$

\begin{proof}[Proof of Proposition \ref{example_cone_3_points}]
Consider $C=V(y^{d-1}z-x^d)\subset\P^2\simeq V(x_0),$ and let $X=V(y^{d-1}z-x^d)\subset\P^3_{[x_0:x:y:z]}$ be the cone over it. Let $R_1=[1:0:0:0]$ be the vertex of the cone, and let $R_2=[1:0:0:1]$ be a point which projects onto the singular point of $C$.
Note that the multiplicity of $R_1$ is $d$ and the multiplicity of $R_2$ is $d-1$.
 
 Consider any $x\in V(x_0).$ If $x\notin C$, the line joining $x$ and $R_1$ does not intersect $V(G)$ besides at $R_1$. If $x\in C$ is a smooth point, the line joining $x$ and $R_2$ does not hit $V(G)$ again. Finally, if $x=[0:0:0:1]$ is the singular point of $C$, then the
 line joining $x$ and, say, $[1:1:0:1]$ does not intersect $V(G)$ again (besides at $x$).  
\end{proof}

\subsection{Uniqueness of the cone construction}
\label{uniqueness_cone_section}

For a polynomial $g\in k[x_1,...,x_n]$ and a point $p\in\A^n_k,$ we say that $g$ vanishes at $p$ with multiplicity $m$ if the polynomial $g(x+p)$ has no terms of total degree less than $m$ but has some nonzero terms of
degree precisely $m$. For a line $l=\{p+tv\ |\ t\in k\}$ passing through $p$, recall that the intersection multiplicity $I_p(l,V(g))$  equals the order of vanishing at $t=0$ of $g(p+tv)\in k[t]$. 

\begin{Lem}
Let $g\in k[x_1,...,x_n]$ be a polynomial of degree $d$ and let $p\in V(g)\subset\A^n_k.$ Let $\G_p\simeq\P^{n-1}$ be the variety of lines in $\A^n_k$ (or, equivalently, in $\P^n_k=\A^n_k\cup V(x_0)$) passing through $p$. Suppose that there exists a 
dense subset $\L\subset\G_p$ such that for each $l\in \L$, we have $I_p(l,V(g))=d.$ Then the multiplicity of the point $p\in V(g)$ is exactly $d$. 
\label{int_mult_singularity_mult}
\end{Lem}      

\begin{proof}
Without loss of generality, $p=(0,...,0)$. Write $g=g_1+g_2+\dots+g_d,$ where $g_i$ is homogeneous of degree $i$. A line through $p$ has the form $\{t(a_1,...,a_{n})\ |\ t\in k\}$ for a uniquely determined $[a_1:...:a_{n}]\in\P^{n-1}$.  Expanding $g$ along such a line, we obtain $g(t(a_1,...,a_{n}))=tg_1(a_1,...,a_n)+t^2g_2(a_1,...,a_n)+\dots+t^d g_d(a_1,...,a_n).$ The given condition now implies that $g_1,...,g_{d-1}$ all vanish on a dense subset $\L$ of $\P^{n-1}.$  Then $\overline{\L}\subset V(g_i)$ for each $i=1,...,d-1$, hence $g_i=0$.  
\end{proof}

Let $p\in V(g)$ be a smooth point on a hypersurface $V(g)$ in $\A^n_k$. Change coordinates so that $p=(0,...,0)$ and expand $g$ near $p$ as $g=g_1+g_2+\dots,$ where $g_i$ is homogeneous of degree $i$
(note that $g_1\neq 0$). We say that $p$ is a flexy point if $g_1|g_2.$ This is a closed condition, so the subset of 
$V(g)_{\text{smooth}}$ consisting of its flexy points is either a proper closed subset of $V(g)_{\text{smooth}}$, or all of $V(g)_{\text{smooth}}$ 
(the latter case can happen only when $k$ has positive characteristic; then $V(g)$ is called a flexy hypersurface). 

An irreducible plane curve $C\subset\P^2_k$ is called ``funny" if there is a point $p_0\in\P^2$ such that for all points $p\in C$, the line joining $p$ and $p_0$ is tangent to $C$ at $p$. If $C$ is irreducible and is
non--funny, then for any point $p_0$, there are only finitely many points $p\in C$ for which the line joining $p$ and $p_0$ is tangent to $C$ at $p$.

\begin{Lem}
Let $C=V(G)\subset \P^2\simeq V(x_0)$ be an irreducible curve, and let $X=V(G)\subset\P^3$ be the cone over it, with vertex $v=[1:0:0:0].$ Let $p$ be a smooth non-flexy point on $C$. 
Then the line joining $p$ and $v$ is the only line $l\subset\P^3$ passing through $p$ which satisfies $I_p(l,X)\geq 3.$
\label{unique_line_int_mult_at_least_3}
\end{Lem}

\begin{proof}
Choose coordinates so that $p=[0:0:0:1]$ and so that the tangent line at $p$ to $C$ in $\P^2$ (or in $\A^2_{x,y}$) is given by $x=0.$ 
Let $g$ be the dehomogenization of $G$ with respect to $z$.
So, $g$ has the form $g=x+\alpha x^2+\beta xy+\gamma y^2+\text{h.o.t.}$ and by assumption, $\gamma\neq 0.$ Given now a line $l$ in $\A^3_{x_0,x,y}$ with $p\in l$, described by
$x_0=ta, x=tb, y=tc$ (for $t\in k$), then we have $I_p(l,V(g))\geq 3$ precisely when $b=c=0, a\neq 0$.
\end{proof}

\begin{Rem}
Let $p\in V(g)$ be a non-flexy smooth point on a hypersurface $V(g)$ in $\A^3_k$ of degree $d$.  Then the number of lines $l\subset\A^3_k$ passing through $p$ with the property that $I_p(l,V(g))\geq 3$ is either one or two. 
\end{Rem}

\begin{proof}[Proof of Proposition \ref{cone_uniqueness_prop}]
We know that $\dim C=1$. 

Let $M=2^{\{R_1,...,R_s\}}$ be the collection of all nonempty subsets of $\{R_1,...,R_s\}$. For each $x\in V(x_0)-C$, choose a Kakeya line $l_x$ through $x$ whose affine part is contained in $E$.  By the pigeonhole principle applied to the assignment $V(x_0)-C\to M$, $x\mapsto l_x\cap V(G),$ there exists a nonempty subset $S\subset\{R_1,...,R_s\}$ which is the image of infinitely many points $x\in V(x_0)-C$. The set $S$ must be singleton, say $S=\{R_1\}$. Let $R_1,...,R_{s'}$ be 
the points among $R_1,...,R_s$  with the property that for each $i=1,...,s',$ there are infinitely many points $x\in V(x_0)-C$ such that $l_x\cap V(G)=\{R_i\}$. For $i=1,...,s',$ let $\Omega_i$ be the set of all $x\in V(x_0)-C$ such that $l_x\cap V(G)=\{R_i\}$. Then the complement of
$\cup_{i=1}^{s'}\Omega_i$ in $V(x_0)-C$ is finite. Thus, $\cup_{i=1}^{s'}\Omega_i$ is dense in $V(x_0)-C$, so for some $i\in\{1,...,s'\},$ the set $\Omega_i$ is dense in $V(x_0)-C$, hence in $V(x_0)$. Say this holds for $i=1,$ and let $R=R_1.$ Then
$I_R(l_x,V(G))=d$ for $x\in\Omega_1$ by Bezout's theorem. By Lemma \ref{int_mult_singularity_mult}, 
$R$ is a point of $V(G)$ of multiplicity $d$. 

For any $x\in C$, the line joining $x$ and $R$ intersects $V(G)$ at $R$ with intersection multiplicity at least $d$, and also intersects $V(G)$ at $x$, hence Bezout's theorem implies that this line is contained in $V(G)$. 
Therefore, the entire cone over $C$ with vertex $R$ is contained in $V(G)$. Since $V(G)$ is irreducible, this implies that $V(G)$ is precisely the cone over $C$ with vertex $R$.
Choose coordinates so that $R=[1:0:0:0]$; then $G$ does not involve the variable $x_0$. Note that
  $C$ is irreducible
(if $C$ were reducible, so would be the cone over it), and $\deg C=\deg V(G)=d.$ 

Next, for any $x\in C$, there exists a Kakeya line $l_x$ passing through $x$, which intersects $V(g)$ only at points among $\{R_2,...,R_s\}$. Since $C$ is non--flexy curve, the set $U$ consisting of all (smooth) non-flexy points of $C$ is an open dense subset of $C$. For any $x\in U$ and any line $l$ through $x$ other than the one joining $x$ and $R$, we know that $I_x(l,V(G))\leq 2$
by Lemma \ref{unique_line_int_mult_at_least_3}; in particular, since $d\geq 3$, the Kakeya line $l_x$ through $x$ must intersect $V(G)$ again. 

Repeat the earlier argument, now for the assignment $U\to 2^{\{R_2,...,R_s\}}, x\mapsto l_x\cap V(g)$ to find a point in 
$\{R_2,...,R_s\}$, say $R_2,$ and a dense subset $\Omega\subset U$, such that for all 
$x\in\Omega,$ we have  $l_x\cap V(G)=\{x,R_2\}$. Let $\overline{R_2}\in C$ be the point of intersection of the line joining $R$ and $R_2$ with $V(x_0)$. Since $C$ is non--funny, there are at most finitely many smooth points $p\in C$ such that the tangent line to $C$ at $p$ passes through $\overline{R_2}$. Shrinking $U$ if necessary, we can assume that for any $x\in U$, the tangent line to $C$ at $x$ in $V(x_0)$ does not pass through $\overline{R_2},$ and for all $x\in\Omega\subset U$, we have 
$l_x\cap V(G)=\{x,R_2\}$. 
For $x\in U$, note that the line joining $x$ and $R_2$ is not contained in the tangent plane to $V(G)$ at $x$.  

For any $x\in\Omega\subset U$, the Kakeya line $l_x$ is the line joining $x$ and $R_2$, so 
$I_x(l_x,V(G))=1.$ By Bezout's theorem, $I_{R_2}(l_x,V(G))=d-1,$ for any $x\in \Omega$. 

We claim that $R_2$ is a point on $V(G)$ of multiplicity $d-1$. Say $R_2=[1:a_1:a_2:a_3]$ and set $a=(a_1,a_2,a_3)\in\A^3_k$ (these are fixed once and for all). 
Note that the dehomogenization of $G$ with respect to the first variable $x_0$ is just $G$ itself.
So, to determine the multiplicity of $R_2$ on $V(G)$, we have to examine
\[G(a+(x_1,x_2,x_3))=g_1(x_1,x_2,x_3)+\dots+ g_d(x_1,x_2,x_3),\]
where $g_i\in k[x_1,x_2,x_3]$ is homogeneous of degree $i$. 
For any $[0:v_1:v_2:v_3]\in \Omega,$ examine the intersection multiplicity at $a$ of $V(G)$ and the line $l_{a,v}$ passing through $a$ in the direction $v=[v_1:v_2:v_3].$ Note that
\[G(a+tv)=tg_1(v_1,v_2,v_3)+t^2g_2(v_1,v_2,v_3)+\dots+t^dg_d(v_1,v_2,v_3).\]
The condition that $I_a(V(G),l_{a,v})=d-1$ means that $g_i(v_1,v_2,v_3)=0$ for $i=1,...,d-2.$ 
 Thus, for $i=1,...,d-2,$ we have $\Omega\subset V(g_i)$ and since $\Omega$ is dense in $C$, we deduce $C\subset V(g_i)$ for all $i=1,...,d-2$. However, $C$ has degree $d$ while $g_i$ has
degree at most $d-2.$ Therefore, $g_i(x_1,x_2,x_3)=0$ in $k[x_1,x_2,x_3]$ for $i=1,...,d-2$. 

Since $R_2$ is a point on $V(G)$ of multiplicity $d-1,$ so is $\overline{R_2}$ on $C$. If $q\neq\overline{R_2}$ is a non-smooth point of $C$, the line joining $\overline{R_2}$ and $q$ would have to be contained in $C$, by Bezout's theorem, and so by irreducibility, $C$ would be a line (but we are in the case $d\geq 3$). Therefore, indeed, $\overline{R_2}$ is the unique singular point of $C$.   
\end{proof}

\begin{Rem} Note that this proof implies that under the assumptions of Proposition \ref{cone_uniqueness_prop}, one must have $s\geq 2$.
So, the construction in the proof of Proposition \ref{example_cone_3_points} is optimal in terms of the number of points $R_i$ that
one needs to add to $\{g\neq 0\}$ to obtain a Kakeya set. 
\end{Rem}

\section{Kakeya subvarieties of affine or projective space}
\label{Kakeya_subs_section}

Let $X\subset\A^N_{x_1,...,x_N}$ be an $n$-dimensional subvariety and let $\overline{X}\subset\P^N_{[x_0:...:x_N]}$ be its Zariski closure. Define $\Delta=\Delta(X)$ as the set of all directions of lines contained in $X$:
\[\Delta(X)=\{v\in V(x_0)\ |\ \text{there exists a line $l\subset X$ such that $\overline{l}\cap V(x_0)=\{v\}$}\}.\]
This is a constructible subset of $V(x_0)$; it is the image under the first projection of
$\{(v,l)\ |\ v\in l\}\subset V(x_0)\times F_1(X),$
where we set
$F_1(X):=(F_1(\overline{X})-F_1(\overline{X}\cap V(x_0)));$
as usual, $F_1$ stands for the Fano variety of a projective variety.

\begin{Prop} Notation as above, we have
$\dim\Delta\leq n-1.$
\end{Prop}

\begin{proof}
For any $v\in\Delta$, there is a line $l\subset X$ whose closure contains $v$. But then $v\in\overline{l}\subset\overline{X}$ and hence $\overline{\Delta}\subset\overline{X}.$ If $\dim\Delta\geq n,$ then $\Delta$ would
have to be an irreducible component of $\overline{X}$, which is impossible since $\Delta\subset V(x_0)$ and $X\subset\A^N$. 
\end{proof}

Note that this proof is again a geometric version of Dvir's argument. 

\begin{Def}
Let $X\subset\A^N$ be an $n$-dimensional subvariety. We say that it is a Kakeya subvariety if the inequality $\dim\Delta\leq n-1$ is an equality. 
\end{Def}

\subsection{Examples coming from hypersurfaces}

Here we give as examples a class of Kakeya varieties. 

\begin{Prop} 
Let $1\leq d\leq N-1$ and let $X=V(f)\subset\A^N$ be a hypersurface of degree $d$. 
Then $X$ is a Kakeya subvariety of $\A^N$. 
\end{Prop}

\begin{proof}
Let $S_d$ be the $k$-vector space of all polynomials in $k[x_1,...,x_N]$ of degree at most $d$, and let $\A(S_d)$ be the affine space associated to $S_d$.  
Consider the incidence correspondence
\[\mathcal{A}=\{(X,v)\in\A(S_d)\times V(x_0)\ |\ \text{some line $l\subset X$ has direction $v$}\}\subset\A(S_d)\times V(x_0).\]
For $v\in V(x_0),$ let $\Sigma_v$ be the fiber over $v$ under the second projection.
Note that $\dim\Sigma_v$ does not depend on the specific $v\in V(x_0)$. 

Let $\G_v^\circ$ be the subset of the Grassmanian $\G(1,N)$ consisting of lines through $v$ and not contained in $V(x_0)$. Consider the incidence correspondence
\[\mathcal{B}=\{(X,l)\in\A(S_d)\times \G_v^\circ\ |\ l\subset \overline{X}\}\subset\A(S_d)\times \G_v^\circ.\]
Under the surjection $\mathcal{B}\to \G_v^\circ,$ each fiber is irreducible of dimension equal to $\dim S_d-d-1$ (see Lemma \ref{hypersurfcontainingaline} below), hence $\mathcal{B}$ is irreducible, of dimension equal to $\dim S_d-d+N-2$. 
On the other hand, the image of $\mathcal{B}$ under the first projection is precisely $\Sigma_v$ (in particular, $\Sigma_v$ is irreducible). 
Let $t:=\dim\mathcal{B}-\dim\Sigma_v$ 
(note that $t$ is independent of $v$), so
$\dim\Sigma_v=\dim S_d-d+N-2-t$. 
Therefore, $\dim\mathcal{A}=\dim S_d-d+2N-3-t$. 

Consider now the map $\phi:\mathcal{A}\to\A(S_d)$. The fiber of $\phi$ over $f\in\A(S_d)$ is the direction set $\Delta(V(f))$, so for each $f\neq 0$ in $\text{Image}(\phi)$ (such $f$'s certainly exist), we have the chain of inequalities
\[\dim\mathcal{A}-\dim S_d\leq \dim\mathcal{A}-\dim(\text{Image}(\phi))\leq\dim\phi^{-1}(f)=\dim(\Delta(V(f)))\leq N-2,\]
hence $t\geq N-d-1.$ On the other hand, we have $t\leq N-d-1$ by Lemma \ref{example_small_t} below, 
and therefore, all inequalities in the above chain must be equalities. 
In particular, for any $f\in\text{Image}(\phi) -\{0\},$ we have $\dim\Delta(V(f))=N-2$. Moreover, $\dim(\text{Image}(\phi))=\dim\A(S_d)$ and since the map $\phi$ is proper (note that $V(x_0)$ is proper over the point, hence so is the basechange map
$\A(S_d)\times V(x_0)\to\A(S_d)$), it is actually surjective. Therefore, in fact, any $f\in\A(S_d)-\{0\}$ belongs to $\text{Image}(\phi)$, and hence the above chain of inequalities holds for any $f\neq 0$. 
\end{proof}

\begin{Lem}
Let $v\in V(x_0)$ and let $l\in \G_v^\circ$. Then $\{X\in\A(S_d)\ |\ l\subset\overline{X}\}$ is an affine space of dimension
$\dim S_d-d-1$.
\label{hypersurfcontainingaline}
\end{Lem}

\begin{proof}
Change coordinates so that $v=[0:1:0:...:0]$ and $l=V(x_2,...,x_N).$ The set under investigation consists of all polynomials in the
ideal $(x_2,...,x_N)$ whose degree is at most $d$. The dimension count follows from inspecting the exact sequence
\[0\to (x_2,...,x_N)_{\leq d}\to k[x_1,...,x_N]_{\leq d}\to k[x_1]_{\leq d}\to 0\]
of $k$-vector spaces.
\end{proof}

\begin{Lem}
Let $v\in V(x_0)$ and let $1\leq d\leq N-1.$ There exists a hypersurface $V(f)\subset\A^N$ of degree $d$ in $\Sigma_v$ such that
\[\dim\{l\in \G_v^\circ\ |\ l-\{v\}\subset V(f)\}=N-d-1.\]
\label{example_small_t}
\end{Lem}

\begin{proof}
Without loss of generality, $v=[0:1:0:...:0]$. Identify $\G_v^\circ\simeq V(x_1)\cap D_+(x_0)\simeq V(x_1)\subset \A^{N}_{x_1,x_2,...,x_N}$ (slightly abusing notation), 
via $l\mapsto l\cap V(x_1).$ Consider hypersurfaces $V(f)$ which contain the line $V(x_2,...,x_N)$
joining $(0,...,0)$ and $v$. Any such $f$ can be written as 
\[f=g_1x_1^{d-1}+g_2x_1^{d-2}+\dots+g_{d-1}x_1+g_d,\]
where 
$g_i\in k[x_2,...,x_N]$,
$\deg(g_i)\leq i$ for each $i$ (with equality holding for some $i$), and each $g_i\in (x_2,...,x_N).$ Then the line joining $v$ and $(0,a_2,...,a_N)$ is contained in $V(f)$ if and only if
$(a_2,...,a_N)\in V(g_1,...,g_d)\subset\A^{N-1}_{x_2,...,x_N}.$ We can certainly pick $g_1,...,g_d$ such that $V(g_1,...,g_d)$ has dimension $N-1-d;$ for example, take $g_i=x_{i+1}$ for $i=1,...,d$. 
\end{proof}

\subsection{Linear projections yield covers}

Next, we investigate the image of a Kakeya subvariety under a liner projection and prove Proposition \ref{Kakeya_sub_gives_cover_Prop}. Note that if $X=\cup X_i$ is the decomposition of $X$ into irreducible components, then 
$\Delta(X)=\cup\Delta(X_i),$ so in the investigation of $\Delta(X),$ when convenient, we can assume that $X$ is irreducible.

Let $X\subset\A^N_{x_1,...,x_N}$ be an irreducible 
$n$-dimensional Kakeya subvariety of $\A^N$, and let $d$ be the degree of its closure $\overline{X}\subset\P^N_{[x_0:...:x_N]}$ in $\P^N$. Let $\Delta\subset V(x_0)$ be the set of directions of lines contained in $X$. Let $P\in V(x_0)-\overline{X}$ be arbitrary, and let
$H\subset\P^N$ be any hyperplane such that $P\notin H$. Let
$\pi:\P^N-\{P\}\to H$
be the linear projection from $P$ to $H$. 

Let $Y=\pi(X)\subset H\cap D_+(x_0)$ and note that $\overline{Y}=\pi(\overline{X}).$ 
By slight abuse of notation, the restriction 
$\pi:\overline{X}\onto \overline{Y}$ is also denoted by $\pi$. Note that $\pi$ is a finite map; in particular, $\dim Y=n$. When $n<N-1,$ we know that $\pi$ is birational and
$\overline{Y}$ has degree $d$ in $H$; when $n=N-1,$ we know that $\overline{Y}=H$ and $\pi$ has degree $d$. 

Since $\pi$ is a linear projection, it induces
$F_1(\overline{X})\to F_1(\overline{Y})$ and $F_1(X)\to F_1(Y).$ 
Let $\Delta'\subset H\cap V(x_0)$ be the set of directions of lines contained in $Y$. Thus, $\pi$ induces also
$\pi:\Delta\to\Delta',$
and all fibers of this map are finite. Therefore, $\dim\Delta'\geq \dim\pi(\Delta)=\dim\Delta=n-1$ and hence $Y=\pi(X)$ is a Kakeya subvariety of $\A^{N-1}=H\cap D_+(x_0).$

We can repeat the process described above and decrease the codimension of $X$ in $\A^N$. When we get to $N=n+1$,
linear projection $\pi$ as above will yield a finite map $\overline{X}\to\P^n=H$ with the property that 
$\pi(\Delta)\subset V(x_0)\cap H$ is $(n-1)$-dimensional. So, $\pi(\Delta)$ will contain an open subset $U\subset V(x_0)\cap H$ and hence for every $v\in U,$ there exists a line $l\subset X$ whose projection passes through $v$.

\subsection{A three--dimensional example coming from the Grassmanian $\G(1,4)$}
\label{G14_section}

Consider the $6$-dimensional Grassmanian $X=\G(1,4)\subset\P(\wedge^2 k^5)\simeq\P^9_{[p_{ij}]},$ embedded in the projective space $\P^9$ as a degree-$5$ subvariety via Plucker coordinates.
 
It is known (for example, see \cite{Rogora}) that $\dim F_1(X)=8$. So, if $W\subset\P^9$ is a $6$-dimensional linear subspace, the expected dimension of $F_1(X\cap W)$ is $2.$ By the main result in \cite{Rogora}, this is one of the examples of $3$-dimensional projective varieties whose Fano variety is $2$-dimensional. 

Consider $W=V(p_{12}-p_{15}, p_{23}-p_{25},p_{34}).$ Then $X'=X\cap W$ is irreducible and of dimension $3$. 
We perform $3$ appropriate linear projections,
whose composition, 
in affine coordinates, is given as follows:
 \[
X'=V(z-xy,ay-b+a,-bx+c)\into\A^6\to \A^3,\quad
(a,b,c,x,y,z)\mapsto (a-x+y,b-z,c).
\]
 
Consider now any direction $[\alpha:1:\gamma],$ with $\alpha\neq 0$. Then
$a=\alpha^2 t, b=\alpha t, c=\alpha\gamma t, x=\gamma, y=\frac{1}{\alpha}-1,z=\gamma(\frac{1}{\alpha}-1)$ defines a line in $\A^6$ which is contained in $X',$ and whose image under $X'\to\A^3$
is a line in direction $[\alpha^2:\alpha:\alpha\gamma]=[\alpha:1:\gamma].$ In other words, the map $X'\to\A^3$ is Kakeya cover in the
sense of Definition \ref{def_linear_Kakeya_cover}.

\section*{Acknowledgments}
This research was performed while the author was visiting the Institute for Pure and Applied Mathematics (IPAM), which is supported by the National Science Foundation.
I thank Terry Tao and Kiran Kedlaya for some inspiring and encouraging discussions.

\end{document}